\documentclass{amsart}
\usepackage[T1]{fontenc}
\usepackage[utf8]{inputenc}

\usepackage[english]{babel}

\usepackage{amssymb}
\usepackage{amsthm}
\usepackage{graphicx}
\usepackage{import}
\usepackage{amsmath}
\usepackage{amscd}
\usepackage{mathtools}
\usepackage{verbatim}
\usepackage[all]{xy}
\usepackage{tikz-cd}
\usepackage{color}
\usepackage{geometry,mathtools}
\newtheorem{thm}{Theorem}[section]
\newtheorem{prop}[thm]{Proposition}
\newtheorem{lem}[thm]{Lemma}
\newtheorem{cor}[thm]{Corollary}
\newtheorem{defi}[thm]{Definition}
\newtheorem{rem}[thm]{Remark}

\renewcommand{\P}{\mathbb{P}}
\newcommand{\Z}{\mathbb{Z}}

\newcommand{\C}{\mathcal{C}}

\title{Rational cubic fourfolds with associated singular K3 surfaces}

\author[H.Awada]{Hanine AWADA}
\address{Institut Montpellierain Alexander Grothendieck \\ %
CNRS \\ %
Universit\'e de Montpellier \\ %
Case Courrier 051 - Place Eug\`ene Bataillon \\ %
34095 Montpellier Cedex 5 \\ %
France}
\email{hanine.awada@umontpellier.fr}

\begin{document}
\begin{abstract}
Generalizing a recent construction of Yang and Yu, we study to what extent one can intersect Hassett's Noether-Lefschetz divisors $\C_d$ in the moduli space of cubic fourfolds $\C$. In particular, we exhibit arithmetic conditions on 20 indexes $d_1,\dots, d_{20}$ that assure that the divisors $\C_{d_1},\dots,\C_{d_{20}}$ all intersect one another. This allows us to produce examples of rational cubic fourfolds with an associated K3 surface with rank 20 N\'eron-Severi group, i.e. a singular K3 surface.
 
\end{abstract}

\maketitle

{}\section{Introduction} \label{sec:introduction}
A lot of the interest around smooth cubic fourfolds is related to their rationality, which is still an open problem for the generic one. The general expectation is that a very general cubic fourfold is nonrational but until now only (few) examples of rational cubic fourfolds are known (see \cite{MR818549}, \cite{MR2390291}, \cite{Kuznetsov_2009}).\\
 
Hassett \cite{hassett_2000} introduced the notion of "\textit{special}" cubic fourfold, that is a cubic that contains an algebraic surface not homologous to a complete intersection. These fourfolds form a countable infinite union of divisors $\C_d$ called \it Noether-Lefschetz divisors \rm inside the moduli space of smooth cubic fourfolds $\C$, which  is a 20-dimensonal quasi-projective variety. Hassett showed that $\C_d$ is irreducible and nonempty if and only if $d>6$ and $d \equiv 0,2 [6]$. Only few $\C_d$ have been defined explicitly in term of surfaces contained in a general element of these divisors (see \cite{MR3934590}, \cite{MR1658216}, \cite{2016arXiv160605321A}, \cite{MR3968870}). More recently, we have deployed similar techniques to study the birational geometry of universal cubics over certain divisors $\C_d$ \cite{AB20}. Hassett \cite{hassett_2000} proved that, for an infinite subset of values of $d$, one can associate a polarized K3 surface $(S,f)$ of degree $d$ to a cubic fourfold in $\C_d$. This is true for $d$ satisfying $4 \not \vert d, 9 \not \vert d$, and $ p \not \vert d$ for any odd prime number $p \equiv 2 [3]$. A natural conjecture, supported by Hassett's Hodge theoretical work (\cite{hassett_2000}, \cite{MR1658216}) and Kuznetsov's derived categorical work \cite{Kuznetsov_2009}, is that any rational cubic fourfold ought to have an associated K3 surface. A link between rationality and the transcendental motive of a cubic fourfold has also been described in \cite{BP20}. For now, every fourfold in $\C_{14}$, $\C_{26}$, $\C_{38}$ and $\C_{42}$ has been proved to be rational (see \cite{MR3968870}, \cite{MR3934590}, \cite{russo2019trisecant}). These cubics all have an associated K3 surface.

\smallskip
The papers \cite{MR1658216} and \cite{2016arXiv160605321A} exhibit examples of rational cubic fourfolds in $\C_8$ (locus of cubic fourfolds containing a plane) and $\C_{18}$ (locus of fourfolds containing an elliptic ruled surface) parametrized by the union of a countably infinite family of codimension two subvarities in $\C$. In fact these fourfolds, fibered respectively in quadric surfaces and sextic del Pezzo surfaces over $\P^2$, are rational whenever the fibration has a rational section. In order to study further these divisors, using lattice theoretic computations,  the paper \cite{MR3238111} studies the intersection of $\C_8$ with $\C_{14}$. The latter divisor corresponds to cubics containing a quartic scroll. This allowed to describe explicitly the intersection in terms of irreducible components and to introduce new examples of rational cubic fourfolds whose associated quadric surface bundle doesn't have a section. More recently, in \cite{awada2019rational} we studied the intersection of $\C_{18}$ with other divisors $\C_{14}, \C_{26}$ and $\C_{38}$ and give new explicit example of rational fourfold whose del Pezzo sextic fibration has no sections.
 
Lately, Yang and Yu \cite{2019arXiv190501936Y} proved that any two Hassett divisors intersect. Moreover, they proved that $\C_{14}, \C_{26}$  and $\C_{38}$ intersect all Hassett divisors. As a consequence, they conclude that every Hassett divisor $\C_d$ contains a union of three codimension two subvarieties in $\C$  parametrizing rational cubic fourfolds (see \cite[Theorem 3.3]{2019arXiv190501936Y}). Actually they are four, since in the meanwhile also $\C_{42}$ has been shown to parametrize rational cubics. However, it seems natural to ask about the intersection of more than two divisors. Using lattice theoretical computations and a result of \cite{2019arXiv190501936Y}, we  prove the following:\\
 
 \medskip
 \textbf{Theorem:} For $3\leq n \leq 20$,  
 \begin{center}
     $\displaystyle{\bigcap_{k=1}^{n}} \C_{d_k} \ne \emptyset$,
 \end{center} 
 for $d_k \geq 8, d_k \equiv 0,2[6]$ and $d_3,..,d_n= 6 \displaystyle{\prod_i} p_i^2$ or $6 \displaystyle{\prod_i} p_i^2+2$ with $p_i$ a prime number.\\
 In particular, up to 20 Hassett divisors in $\C$ have non-empty intersection.\\
 
 \medskip
 
 These 20 divisors can be all chosen in a way such that their cubics have associated polarized K3 surfaces. This implies that cubics in the intersection of the 20 divisors have an associated K3 with N\'{e}ron-Severi group of rank 20. 
As it is well-known \cite{shioda_inose_1977} K3 surfaces $S$ with Picard number $\rho(S)=rank(NS(S))=20$, the maximal possible, are called \it singular K3 surfaces. \rm Moreover, a special feature of this set of  K3 surfaces is that it is bijective to the set of equivalence classes of positive-definite even integral
binary quadratic forms with respect to $\mathcal{SL}_2(\mathbb{Z})$.
 
 
 \smallskip
\textbf{Plan of the paper:} We first introduce in \S 2 some lattice theory related to cubic fourfolds. Then, in \S 3, by lattice computations, we prove the nonemptiness of the intersection of certain divisors $\C_d$ in $\C$ and give examples of rational cubic fourfold with associated singular K3 surfaces.

\smallskip Computations were done using Macaulay2 \cite{M2} and Sagemath \cite{sagemath}.\\
 

\smallskip
\textbf{Acknowledgement:} It is a pleasure to thank my PhD advisor Michele Bolognesi for his permanent support.
\section{Lattice theory for cubic fourfolds } 
Let $X$ be a smooth cubic fourfold. Next, we will set some notation for latter use.\\

Let $E_8^1$ and $E_8^2$ be two copies of $E_8$, the unimodular positive definite even rank 8 lattice associated to the corresponding Dynkin diagram represented by the following matrix in the basis $<t_k^i>$, for $k=1,..,8$, $i=1,2$:

\medskip
\begin{equation*}
E_8= 
\begin{pmatrix}
2&-1&0&0&0&0&0&0 \\
-1 &2& -1 &0&0&0&0&0\\
0&-1&2&-1&-1&0&0&0\\
0&0&-1&2&0&0&0&0\\
0&0&-1&0&2&-1&0&0\\
0&0&0&0&-1&2&-1&0\\
0&0&0&0&0&-1&2&-1\\
0&0&0&0&0&0&-1&2
\end{pmatrix}
\end{equation*} \\
In addition, let $U^1$ and $U^2$ be two copies of the hyperbolic plane  $U = \begin{pmatrix}
0 & 1 \\
1 & 0 \\
\end{pmatrix}$.
 The basis of $U^k$ consists of vectors $e^k_1$ and $e^k_2$, $k =1, 2$  such that 
 
\medskip
\begin{center}
    $<e^k_1,e^k_1>=0, <e^k_1,e^k_2>=1, <e^k_2,e^k_2>=0$.
\end{center}
\medskip
Let $a_1$ and $a_2$ be the generators of $A_2 = \begin{pmatrix}
2 & 1 \\
1 & 2 \\
\end{pmatrix}$ such that 

\medskip
\begin{center}
 $<a_i,a_i>=2, <a_i,a_j>=1$, $\forall i, j =1, 2$.  
\end{center}
\medskip
Using previous notations, we denote by $L$ the middle integral cohomology lattice of $X$ as follows

\medskip
$$
    \begin{aligned}
       L:= H^4(X,\Z) &\simeq& E_8^{1} \bigoplus E_8^{2} \bigoplus U^{1}\bigoplus U^{ 2} \bigoplus I_{3,0}\\ &\simeq& (+1)^{\bigoplus 21} \bigoplus (-1)^{\bigoplus 2},
    \end{aligned}
$$\\
and $L_0$ the primitive cohomology lattice of signature 22 such that

\medskip
$$ \begin{aligned}
    L_0:= H^4(X,\Z)_0 := <h^2>^{\perp} &\simeq& E_8^{1} \bigoplus E_8^{2} \bigoplus U^{1}\bigoplus U^{ 2} \bigoplus A_2
\end{aligned}$$\\
with $h$ the hyperplane class defined by the embedding of $X$ is $\P^5$:

\begin{center}
    $h^2 =(1,1,1) \in I_{3,0}$.
\end{center}

\medskip
Finally, knowing that the integral Hodge conjecture holds for $X$, we denote by $A(X)=H^4(X,\Z) \cap H^{2,2}(X)$ the lattice of algebraic 2-cycles on $X$ up to rational equivalence.

\medskip
\begin{defi}
The \textit{labelling} of a special cubic fourfold consists of a positive definite rank two saturated sublattice $K$, with $h^2 \in K \subset A(X)$. 
\end{defi}

The discriminant $d$ of a labelling  $K_d =\Z h^2+\Z S$, with $S$ an algebraic surface not homologous to $h^2$, is the determinant of the intersection form on $K_d$:

\medskip
$$
K_d= \begin{tabular}{ l| c r}
& $h^2$ & $S$ \\
\hline
$h^2$ & 3 & $<h^2,S>$\\
$S$ & $<h^2,S>$ & $<S,S>$\\
\end{tabular}
$$

\medskip
$\C_d$ denotes the divisor of special cubic fourfolds with a labelling of discriminant $d$. It is the closure of the locus of cubic fourfolds containing $S$. 

\section{Intersection of divisors}

In this section, we will study the intersection of Hassett divisors $\C_d \in \C$. In particular, we will set the discriminants $d_k$ for which the intersection $\displaystyle{\bigcap_k \C_{d_k}}$ is surely nonempty.

\begin{prop}\cite[Theorem 1.0.1]{hassett_2000}
The Hassett divisor $\C_d \in \C$ is an irreducible and nonempty divisor if and only if 
\begin{center}
    $d\geq8$ and $d \equiv 0,2 [6]$ $(*)$
\end{center} 
\end{prop}

\begin{prop}\cite[Theorem 3.1]{2019arXiv190501936Y} \label{YY}
 Any two Hassett divisors intersect i.e $\C_{d_1} \cap \C_{d_2} \ne \emptyset$ for any integers $d_1$ and $d_2$ satisfying $(*)$
\end{prop}
 
 \begin{rem}
 If $X$ is a cubic fourfold in $\C_{d_1} \cap \C_{d_2}$, then there exists infinite divisors $\C_d$ containing $X$. In fact, $A(X)$ is a rank 3 lattice generated by $h^2$ a rank 2 sublattice $(h^2)^\perp$. For every vector $u$ in the sublattice $(h^2)^\perp$, we have that $u^\perp\subset A(X)$ is of rank 2 and of course contains $h^2$. This defines a labelling, that correspeonds to a divisor containing $X$. Moving $u$ appropriately we obtain as many different labellings as we want.
 \end{rem}
\medskip
We denote by $(**)$ the following condition on the discriminant $d$:
\begin{center}
    $d=6\displaystyle{\prod_i}p_i^2$ or $d=6\displaystyle{\prod_i}p_i^2+2$, where $p_i$ is a prime number (not necessary distinct)
\end{center} 

The proof of the main theorem of this paper will be divided into different steps.

\begin{lem}\label{lem3}
We have $\C_{d_1} \cap \C_{d_2} \cap \C_{d_3} \ne \emptyset$ for any integers $d_1, d_2, d_3$ satisfying (*) such that $d_3$ satisfies (**).
\end{lem}
\begin{proof}
The proof is divided in multiple cases depending on the value of the discriminants.

\medskip
\underline{\textbf{Case 1}}: if $d_1, d_2, d_3 \equiv 0 [6]$ such that $d_1=6n_1, d_2=6n_2$ and $d_3=6n_3$, for $n_1, n_2 \ge 2$ and $n_3= \displaystyle{\prod_i} p_i^2$ with $p_i$ a prime number.\\

We consider the rank 4 lattice,\\

$h^2 \in M := <h^2,\alpha_1,\alpha_2,\alpha_3>$, with $\alpha_1 =e^1_1+n_1e^1_2, \alpha_2=e^2_1+n_2e^2_2$ and $\alpha_3= \sqrt{n_3}a_1$.\\

The Gram matrix of $M$ with respect to this basis is:
$$\begin{pmatrix}
3&0&0 & 0 \\!
0&2n_1&0& 0 \\
0&0&2n_2&0\\
0&0&0&2n_3\\
\end{pmatrix}$$
Therefore, $M$ is a positive definite saturated sublattice of $L$.\\
In addition, for any nonzero $v=x_1h^2+x_2\alpha_1+x_3\alpha_2+x_4\alpha_3$, where $x_1,..,x_4$ are integers not all zeros, we have
\begin{center}
    $<v,v>=3x_1^2+2n_1x_2^2+2n_2x_3^2+2n_3x_4^2 \geq 3$,
\end{center}
$\forall n_1,n_2 \geq 2, n_3$ as required.

\smallskip
We will write $\C_M \subset \C$ for the locus of smooth cubic fourfold such that there is a primitive embedding $M \subset A(X)$ preserving $h^2$.\\ In our case, $\C_M$ is nonempty of codimension 3 by \cite[Proposition 2.3]{2019arXiv190501936Y}. Moreover, we consider 
\begin{center}
$h^2 \in K_{d_1}:= <h^2, \alpha_1> \subset M$\\
$h^2 \in K_{d_2}:= <h^2, \alpha_2> \subset M$\\
$h^2 \in K_{d_3}:= <h^2, \alpha_3> \subset M$
\end{center}

of discriminant respectively $d_1,d_2$ and $d_3$.

By \cite[Lemma 2.4, Proposition 2.3]{2019arXiv190501936Y}, we have that $\emptyset \ne \C_M \subset \C_{d_1} \cap \C_{d_2} \cap \C_{d_3}$ as required.

\medskip
\underline{\textbf{Case 2}:} if $d_1, d_2\equiv 0 [6], d_3 \equiv 2 [6]$ such that $d_1=6n_1, d_2=6n_2$ and $d_3=6n_3+2$  for $ n_1,n_2 \geq 2$, and $n_3= \displaystyle{\prod_i} p_i^2$ with $p_i$ a prime number.\\

We consider, this time, the positive definite saturated rank 4 sublattice of $L$:
$$h^2 \in M := <h^2,\alpha_1,\alpha_2,\alpha_3+(0,0,1)>.$$

The Gram matrix of $M$ is:
$$\begin{pmatrix}
3&0&0 & 1 \\
0&2n_1&0& 0 \\
0&0&2n_2&0\\
1&0&0&2n_3+1\\
\end{pmatrix}$$

Additionally, for any nonzero $v=x_1h^2+x_2\alpha_1+x_3\alpha_2+x_4(\alpha_3+(0,0,1))$ we have that:
$$
\begin{aligned}
    <v,v>&=&3x_1^2+2n_1x_2^2+2n_2x_3^2+(2n_3+1)x_4^2+2x_1x_4\\&=& 2x_1^2+2n_1x_2^2+2n_2x_3^2+2n_3x_4^2+(x_1+x_4)^2 \geq 3
\end{aligned}$$
 $\forall n_1,n_2 \geq 2$ and $n_3=\displaystyle{\prod_i}p_i^2$.
Then, applying \cite[Proposition 2.3]{2019arXiv190501936Y}, we have that $\C_M$ is nonempty of codimension 3.\\
Moreover, we consider the saturated rank 2 sublattices of $M$:\\ 
\begin{center}
$h^2 \in K_{d_1}:= <h^2, \alpha_1>$ with Gram matrix
$\begin{pmatrix}
3&0\\
0&2n_1
\end{pmatrix}$\\

\medskip
$h^2 \in K_{d_2}:= <h^2, \alpha_2>$ with Gram matrix 
$\begin{pmatrix}
3&0\\
0&2n_2
\end{pmatrix}$\\

\medskip
\hspace{2cm}$h^2 \in K_{d_3}:= <h^2, \alpha_3+(0,0,1)>$  with Gram matrix 
$\begin{pmatrix}
3&1\\
1&2n_3+1
\end{pmatrix}$
\end{center}
Again, by \cite[Lemma 2.4]{2019arXiv190501936Y} and \cite[Proposition 2.3]{2019arXiv190501936Y}, we have that $\C_M \subset \C_{K_{d_1}}=\C_{d_1}$, $\C_M \subset \C_{K_{d_2}}=\C_{d_2}$ and $\C_M \subset \C_{K_{d_3}}=\C_{d_3}$. Consequently, $\C_{d_1} \cap \C_{d_2} \cap \C_{d_3} \ne \emptyset$ in this case as well.\\

\underline{\textbf{Case 3}:} if $d_1 \equiv 0 [6]$ and $d_2, d_3 \equiv 2 [6]$ such that $d_1=6n_1, d_2=6n_2+2$ and $d_3=6n_3+2$  for $n_1\geq 2, n_2 \geq 1$ and $n_3= \displaystyle{\prod_i} p_i^2$ with $p_i$ a prime number.\\

We consider the rank 4 lattice,

\begin{center}
$h^2 \in M := <h^2,\alpha_1,\alpha_2+ (0,1,0),\alpha_3+(0,0,1)>$
\end{center}

The Gram matrix of $M$ is:
$$\begin{pmatrix}
3&0&1 & 1 \\
0&2n_1&0& 0 \\
1&0&2n_2+1&0\\
1&0&0&2n_3+1\\
\end{pmatrix}$$
$M$ is a positive definite saturated sublattice of $L$. Moreover, for any nonzero $v=x_1h^2+x_2\alpha_1+x_3(\alpha_2+(0,1,0))+x_4(\alpha_3+(0,0,1)) \in M$,
$$
\begin{aligned}
    <v,v>&=&3x_1^2+2n_1x_2^2+(2n_2+1)x_3^2+(2n_3+1)x_4^2+2x_1x_3+2x_1x_4\\
    &=& x_1^2+2n_1x_2^2+2n_2x_3^2+2n_3x_4^2+(x_1+x_3)^2+(x_1+x_4)^2 \geq 3
\end{aligned}
$$
 $\forall n_1\geq 2, n_2 \geq 1,  n_3$ as required.
Then $\C_M$ is nonempty of codimension 3 by \cite[Proposition 2.3]{2019arXiv190501936Y}. In addition, we consider the saturated rank 2 sublattices of $M$:
\begin{center}
$h^2 \in K_{d_1}:= <h^2, \alpha_1>$ of Gram matrix
$\begin{pmatrix}
3&0\\
0&2n_1
\end{pmatrix}$\\

\hspace{2cm}$h^2 \in K_{d_2}:= <h^2, \alpha_2+(0,1,0)>$ of Gram matrix
$\begin{pmatrix}
3&1\\
1&2n_2+1
\end{pmatrix}$\\

\hspace{2cm}$h^2 \in K_{d_3}:= <h^2, \alpha_3+(0,0,1)>$  of Gram matrix 
$\begin{pmatrix}
3&1\\
1&2n_3+1
\end{pmatrix}$
\end{center}

Therefore, $\C_{d_1} \cap \C_{d_2} \cap \C_{d_3} \ne \emptyset$ in this case.\\

\underline{\textbf{Case 4}:} if $d_1, d_2, d_3 \equiv 2 [6]$ such that $d_1=6n_1+2, d_2=6n_2+2$ and $d_3=6n_3+2$ for  $n_1, n_2 \geq 1$ and $n_3= \displaystyle{\prod_i} p_i^2$ with $p_i$ a prime number.\\

We consider the lattice $M$ as follows

$$h^2 \in M := <h^2,\alpha_1+(1,0,0),\alpha_2+ (0,1,0),\alpha_3+(0,0,1)>$$

The Gram matrix of $M$ with respect to this basis is:
$$\begin{pmatrix}
3&1&1 & 1 \\
1&2n_1+1&0& 0 \\
1&0&2n_2+1&0\\
1&0&0&2n_3+1\\
\end{pmatrix}$$
$M$ is a rank 4 positive definite saturated sublattice of $L$.\\
Further, for any nonzero $v=x_1h^2+x_2(\alpha_1+(1,0,0))+x_3(\alpha_2+(0,1,0))+x_4(\alpha_3+(0,0,1)) \in M$, we have that:
$$\begin{aligned}
    <v,v>&=&3x_1^2+(2n_1+1)x_2^2+(2n_2+1)x_3^2+(2n_3+1)x_4^2+2x_1x_2+2x_1x_3+2x_1x_4\\&=& 2n_1x_2^2+2n_2x_3^2+2n_3x_4^2+(x_1+x_2)^2(x_1+x_3)^2+(x_1+x_4)^2 \geq 3
\end{aligned}$$
 $\forall n_1, n_2 \geq 1,  n_3$ as required. We consider rank 2 saturated sublattices of $M$:

\begin{center}
$h^2 \in K_{d_1}:= <h^2, \alpha_1+(1,0.0)>$\\
$h^2 \in K_{d_2}:= <h^2, \alpha_2+(0,1,0)>$\\
$h^2 \in K_{d_3}:= <h^2, \alpha_3+(0,0,1)>$
\end{center} 
Applying \cite[Lemma 2.4]{2019arXiv190501936Y} and \cite[Proposition 2.3]{2019arXiv190501936Y}, $\C_{d_1} \cap \C_{d_2} \cap \C_{d_3} \ne \emptyset$ in this case.\\
\end{proof}

\begin{rem}
The condition required on the third discriminant is a sufficient condition not necessary.
\end{rem}
\begin{rem}
Yang and Yu \cite[Theorem 3.3]{2019arXiv190501936Y} prove that $\C_{d_1} \cap \C_{d_2} \cap \C_{14} \ne \emptyset$ for $d_1$ and $d_2$ satisfying $(*)$.
\end{rem}

\begin{prop} \label{prop20}
 $\displaystyle{\bigcap_{k=1}^{20}} \C_{d_k} \ne \emptyset$, $\forall d_k \equiv 0[6]$ such that $d_3,..d_{20}=6 \displaystyle{\prod_i}p_i^2$.
\end{prop}
\begin{proof}
One can see that $\forall n \geq 4$, $\displaystyle{\bigcap_{k=1}^{n}} \C_{d_k} \ne \emptyset$.\\
Note that for $n \leq 3$, this formula is already proved (see Proposition \ref{YY}  and Lemma \ref{lem3}). We proceed next with the following cases.\\

\underline{For $n=4$:}  $d_1=6n_1, d_2=6n_2, d_3=6n_3, d_4=6n_4$ for $n_1,n_2 \geq 2$ and $n_3, n_4=\displaystyle{\prod_i} p_i^2$ with $p_i$ a prime number.\\

We consider the rank 5 lattice,\\

$h^2 \in M := <h^2,\alpha_1,\alpha_2,\alpha_3, \alpha_4>$, with $\alpha_1 =e^1_1+n_1e^1_2, \alpha_2=e^2_1+n_2e^2_2$, $\alpha_3= \sqrt{n_3}a_1$ and $\alpha_4=\sqrt{n_4}a_2$\\

$M$ is a positive definite saturated sublattice of $L$ with the following Gram matrix: $$\begin{pmatrix}
3&0&0 & 0&0\\
0&2n_1&0& 0&0 \\
0&0&2n_2&0&0\\
0&0&0&2n_3&\sqrt{n_3n_4}\\
0&0&0&\sqrt{n_3n_4}&2n_4\\
\end{pmatrix}.$$
Furthermore, for any nonzero $v=x_1h^2+x_2\alpha_1+x_3\alpha_3+x_4\alpha_3+x_5\alpha_4 \in M$, we have that
\begin{center}
    $<v,v>=3x_1^2+2n_1x_2^2+2n_2x_3^2+n_3x_4^2+n_4x_5^2 + (\sqrt{n_3}x_4+\sqrt{n_4}x_5)^2 \geq 3$
\end{center}
$\forall n_1,n_2 \geq 2$, $n_3, n_4= \displaystyle{\prod_i} p_i^2$ with $p_i$ a prime number (thus $n_3, n_4 \geq 4$).
$\C_M$ is then nonempty of codimension 4. Moreover, we consider rank 2 saturated sublattices of $M$:
$$h^2 \in K_{d_i}:= <h^2, \alpha_i>, i:1,..,4$$
 of discriminant respectively $d_i$.
 By \cite[Lemma 2.4]{2019arXiv190501936Y} and \cite[Proposition 2.3]{2019arXiv190501936Y}, We have that $\emptyset \ne \C_M \subset \C_{d_1} \cap \C_{d_2} \cap \C_{d_3} \cap \C_{d_4}$ with the required conditions on $d_i$.

\medskip
We keep checking the non emptiness of $\displaystyle{\bigcap_{k=1}^{n}} \C_{d_k}$, for $5\leq n \leq 19$, using the same method. We leave details of these steps to the interested readers, however we mention last step.\\

\medskip
\underline{For n=20}: if $d_i=6n_i$, $i=1,..,20$, with $n_1,n_2 \geq 2$ and $n_3,..,n_{20}=\displaystyle{\prod_i} p_i^2$ with $p_i$ a prime number.\\we consider the rank 21 lattice $M$ generated by:
\begin{center}

$ < h^2,\alpha_1,\alpha_2,\sqrt{n_3}a_1,\sqrt{n_4}a_2,\sqrt{n_5}t^1_1,\sqrt{n_6}t^1_3,\sqrt{n_7}t^1_6,\sqrt{n_8}t^2_1,\sqrt{n_9}t^2_3,\sqrt{n_{10}}t^2_6,\sqrt{n_{11}}t^1_2,\sqrt{n_{12}}t^2_2,\sqrt{n_{13}}t^1_4$,\\ $\sqrt{n_{14}}t^2_4, \sqrt{n_{15}}t^1_7,\sqrt{n_{16}}t^2_7,\sqrt{n_{17}}t^1_8,\sqrt{n_{18}}t^2_8,\sqrt{n_{19}}t^1_5,\sqrt{n_{20}}t^2_5>$\end{center}

with the following Gram matrix:

$$\left(\begin{array}{ccc|ccc}
{}&A&{}&{}&B&{}\\
{}&{}&{}&{}&{}{}\\
\hline
{}&{}&{}&{}&{}{}\\
{}&B^T&{}&{}&C&{}
\end{array}\right)$$
such that

$$A=\begin{pmatrix}
3&0&0 & 0&0&0&0&0&0&0&0\\
0&2n_1&0& 0&0&0&0&0&0&0&0 \\
0&0&2n_2&0&0&0&0&0&0&0&0\\
0&0&0&2n_3&\sqrt{n_3n_4}&0&0&0&0&0&0\\
0&0&0&\sqrt{n_3n_4}&2n_4&0&0&0&0&0&0\\
0&0&0&0&0&2n_5&0&0&0&0&0\\
0&0&0&0&0&0&2n_6&0&0&0&0\\
0&0&0&0&0&0&0&2n_7&0&0&0\\
0&0&0&0&0&0&0&0&2n_8&0&0\\
0&0&0&0&0&0&0&0&0&2n_9&0\\
0&0&0&0&0&0&0&0&0&0&2n_{10}\\
\end{pmatrix}$$

\vspace{1cm}
\begin{center}
\scalebox{0.9}
{$B=\begin{pmatrix} 
0&0&0&0&0&0&0&0&0&0\\
0&0&0&0&0&0&0&0&0&0\\
0&0&0&0&0&0&0&0&0&0\\
0&0&0&0&0&0&0&0&0&0\\
0&0&0&0&0&0&0&0&0&0\\
-\sqrt{n_{5}n_{11}}&0&0&0&0&0&0&0&0&0\\
-\sqrt{n_{6}n_{11}}&0&-\sqrt{n_6n_{13}}&0&0&0&0&0&-\sqrt{n_6n_{19}}&0\\
0&0&0&0&-\sqrt{n_7n_{15}}&0&0&0&-\sqrt{n_7n_{19}}&0\\
0&-\sqrt{n_8n_{12}}&0&0&0&0&0&0&0&0\\
0&-\sqrt{n_9n_{12}}&0&-\sqrt{n_9n_{14}}&0&0&0&0&0&-\sqrt{n_9n_{20}}\\
0&0&0&0&0&-\sqrt{n_{10}n_{16}}&0&0&0&-\sqrt{n_{10}n_{20}}\\
\end{pmatrix}$}\end{center}

\vspace{1cm}
$$C=
\begin{pmatrix}
2n_{11}&0&0&0&0&0&0&0&0&0\\
0&2n_{12}&0&0&0&0&0&0&0&0\\
0&0&2n_{13}&0&0&0&0&0&0&0\\
0&0&0&2n_{14}&0&0&0&0&0&0\\
0&0&0&0&2n_{15}&0&-\sqrt{n_{15}n_{17}}&0&0&0\\
0&0&0&0&0&2n_{16}&0&-\sqrt{n_{16}n_{18}}&0&0\\
0&0&0&0&-\sqrt{n_{15}n_{17}}&0&2n_{17}&0&0&0\\
0&0&0&0&0&-\sqrt{n_{16}n_{18}}&0&2n_{18}&0&0\\
0&0&0&0&0&0&0&0&2n_{19}&0\\
0&0&0&0&0&0&0&0&0&2n_{20}\\
\end{pmatrix}
$$\\\\
$h^2 \in M \subset L$ is a positve definite saturated lattice. Moreover, for all $v \in M$, we have that

\begin{center}
    $<v,v>\ne 2$.
\end{center}
for $n_1,n_2 \geq 2$, $n_3,..,n_{20}$ as required. Then, by \cite[Propsition 2.3, Lemma 2.4]{2019arXiv190501936Y}, $\C_M$ is nonempty of codimension 20.\\ Moreover, we consider the rank 2 saturated sublattices $K_{d_i} \subset M$ generated by $h^2$ and another element from the basis of $M$: $\alpha_1,\alpha_2,\sqrt{n_3}a_1,\sqrt{n_4}a_2,\sqrt{n_5}t^1_1,\sqrt{n_6}t^1_3,\sqrt{n_7}t^1_6,\sqrt{n_8}t^2_1,\sqrt{n_9}t^2_3,\\\sqrt{n_{10}}t^2_6,\sqrt{n_{11}}t^1_2,\sqrt{n_{12}}t^2_2,\sqrt{n_{13}}t^1_4$, $\sqrt{n_{14}}t^2_4, \sqrt{n_{15}}t^1_7,\sqrt{n_{16}}t^2_7,\sqrt{n_{17}}t^1_8,\sqrt{n_{18}}t^2_8,\sqrt{n_{19}}t^1_5,\sqrt{n_{20}}t^2_5$. $K_{d_i}$ are of discriminant $d_i$ respectively. Applying \cite[Propsition 2.3, Lemma 2.4]{2019arXiv190501936Y} again, we obtain $\emptyset \ne \C_M \subset \C_{d_1} \cap .. \cap \C_{d_{20}}$ with the required conditions on $d_i$.
\end{proof}

In order to prove the result for a more general case, let us show now that $\displaystyle{\bigcap_k \C_{d_k}} \ne \emptyset$ for $d_k$ satisfying $(*)$ and $d_3,..,d_{20}$ satisfy $(**)$ also.

\begin{thm}\label{thm}
We have $\displaystyle{\bigcap_{k=1}^{20}} \C_{d_k} \ne \emptyset$ for $d_k$ satisfying $(*)$ and $d_3,..,d_{20}= 6 \displaystyle{\prod_i} p_i^2$ or $6 \displaystyle{\prod_i} p_i^2+2$, with $p_i$ a prime number.
\end{thm}

\begin{proof}
Let us consider $\displaystyle{\bigcap_{k=1}^{n}} \C_{d_k}$.\\
\underline{For $n=2$:} see \cite[Theorem 3.1]{2019arXiv190501936Y}.\\
\underline{For $ n= 3$}: see Lemma \ref{lem3}.\\
\underline{For $1\leq n \leq 20$}:  for $d_k \equiv 0 [6]$, $d_3,..,d_{20}=6 \displaystyle{\prod_i} p_i^2$, see Proposition \ref{prop20}.\\
In the following, we regroup some steps to prove the Theorem \ref{thm}. Details are left to the interested readers.\\
\underline{For $n=4$:}

\medskip
\underline{\textbf{Case 1}:} proved before in Proposition \ref{prop20}.

\medskip
\underline{\textbf{Case 2}:} if $d_1=6n_1, d_2=6n_2, d_3=6n_3, d_4=6n_4+2$ such that $n_1, n_2 \geq 2$ and $n_3, n_4= \displaystyle{\prod_i} p_i^2$ with $p_i$ a prime number.\\

We consider the rank 5 lattice,\\

$h^2 \in M := <h^2,\alpha_1,\alpha_2,\alpha_3, \alpha_4+(0,0,1)>$, with $\alpha_1 =e^1_1+n_1e^1_2, \alpha_2=e^2_1+n_2e^2_2$, $\alpha_3= \sqrt{n_3}a_1$ and $\alpha_4=\sqrt{n_4}a_2$.\\

The Gram matrix of $M$ with respect to this basis is:
$$\begin{pmatrix}
3&0&0 & 0&1\\
0&2n_1&0& 0&0 \\
0&0&2n_2&0&0\\
0&0&0&2n_3&\sqrt{n_3n_4}\\
1&0&0&\sqrt{n_3n_4}&2n_4+1\\
\end{pmatrix}.$$
$M$ is a positive definite saturated sublattice of $L$.\\
For all $v \in M$ $v=x_1h^2+x_2\alpha_1+x_3\alpha_2+x_4\alpha_3+x_5(\alpha_4+(0,0,1)) \in M$, we have that
\begin{center}
    $<v,v>=2x_1^2+2n_1x_2^2+2n_2x_3^2+n_3x_4^2+n_4x_5^2 + (\sqrt{n_3}x_4+\sqrt{n_4}x_5)^2 + (x_1+x_5)^2 \ne 2$
\end{center}
for all $n_1,n_2\geq 2$, $n_3, n_4=\displaystyle{\prod_i} p_i^2$ with $p_i$ a prime number (thus $n_3, n_4 \geq 4$).
Then $\C_M$ is nonempty of codimension 4. Moreover, we consider \\
$$h^2 \in K_{d_i}:= <h^2, \alpha_i>, i:1,..,4.$$
 Hence, by \cite[Lemma 2.4, Proposition 2.3]{2019arXiv190501936Y}, we have that $\emptyset \ne \C_M \subset \C_{d_1} \cap \C_{d_2} \cap \C_{d_3} \cap \C_{d_4}$ with the required conditions on $d_i$ in this case.
 
\medskip
\underline{\textbf{Case 3:}}
if $d_1=6n_1, d_2=6n_2, d_3=6n_3+2, d_4=6n_4+2$ such that $n_1,n_2\geq 2$ and $n_3, n_4= \displaystyle{\prod_i} p_i^2$ with $p_i$ a prime number.\\

We consider the rank 5 lattice,

$$h^2 \in M := <h^2,\alpha_1,\alpha_2,\alpha_3+(0,1,0), \alpha_4+(0,0,1)>,$$ 

with $\alpha_1 =e^1_1+n_1e^1_2, \alpha_2=e^2_1+n_2e^2_2$, $\alpha_3= \sqrt{n_3}a_1$ and $\alpha_4=\sqrt{n_4}a_2$.\\

The Gram matrix of $M$ is:
$$\begin{pmatrix}
3&0&0 & 1&1\\
0&2n_1&0& 0&0 \\
0&0&2n_2&0&0\\
1&0&0&2n_3+1&\sqrt{n_3n_4}\\
1&0&0&\sqrt{n_3n_4}&2n_4+1\\
\end{pmatrix}$$
$M$ is a positive definite saturated sublattice of $L$.\\
In addition, for all $v \in M$ $v=x_1h^2+x_2\alpha_1+x_3\alpha_2+x_4(\alpha_3+(0,1,0))+x_5(\alpha_4+(0,0,1)) \in M$, we have that
\begin{center}
    $<v,v>=x_1^2+2n_1x_2^2+2n_2x_3^2+n_3x_4^2+n_4x_5^2 + (\sqrt{n_3}x_4+\sqrt{n_4}x_5)^2 + (x_1+x_5)^2+ (x_1+x_4)^2\ne 2$
\end{center}
for all $n_1,n_2\geq 2$, $n_3, n_4=\displaystyle{\prod_i} p_i^2$ with $p_i$ a prime number. Then $\C_M$ is nonempty of codimension 4 by \cite[Proposition 2.3]{2019arXiv190501936Y}. Moreover, we consider \\
$$h^2 \in K_{d_i}:= <h^2, \alpha_i>, i:1,..,4.$$\\
By \cite[Lemma 2.4, Proposition 2.3]{2019arXiv190501936Y}, we have that $\emptyset \ne \C_M \subset \C_{d_1} \cap \C_{d_2} \cap \C_{d_3} \cap \C_{d_4}$ with the required conditions on $d_i$ in this case.

\medskip
\underline{\textbf{Case 4:}} if
 $d_1=6n_1, d_2=6n_2+2, d_3=6n_3+2, d_4=6n_4+2$ such that $n_1\geq 2$, $n_2 \ge 1$ and $n_3, n_4= \displaystyle{\prod_i} p_i^2$ with $p_i$ a prime number.\\

We consider the rank 5 lattice,\\
$$h^2 \in M := <h^2,\alpha_1,\alpha_2+(1,0,0),\alpha_3+(0,1,0), \alpha_4+(0,0,1)>.$$

The Gram matrix of $M$ is:
$$\begin{pmatrix}
3&0&1 & 1&1\\
0&2n_1&0& 0&0 \\
1&0&2n_2+1&0&0\\
1&0&0&2n_3+1&\sqrt{n_3n_4}\\
1&0&0&\sqrt{n_3n_4}&2n_4+1\\
\end{pmatrix}.$$
Hence, $M$ is a positive definite saturated sublattice of $L$.\\
Furthermore, for all $v \in M$, $v=x_1h^2+x_2\alpha_1+x_3(\alpha_2+(1,0,0))+x_4(\alpha_3+(0,1,0))+x_5(\alpha_4+(0,0,1)) \in M$, we have that
\begin{center}
    $<v,v>=2n_1x_2^2+2n_2x_3^2+n_3x_4^2+n_4x_5^2 + (\sqrt{n_3}x_4+\sqrt{n_4}x_5)^2 + (x_1+x_5)^2+ (x_1+x_4)^2 +(x_1+x_3)^2\ne 2$
\end{center}
for all $n_1\ge2$, $n_2\geq1$ and $n_3, n_4= \displaystyle{\prod_i} p_i^2$ with $p_i$ a prime number (thus $n_3, n_4 \geq 4$).
Then $\C_M$ is nonempty of codimension 4. Moreover, we consider
$$h^2 \in K_{d_i}:= <h^2, \alpha_i>, i:1,..,4.$$
 We have then that $\emptyset \ne \C_M \subset \C_{d_1} \cap \C_{d_2} \cap \C_{d_3} \cap \C_{d_4}$ with the required conditions on $d_i$ in this case.
 
 \medskip
 \underline{\textbf{Case 5:}}
 if $d_1=6n_1+2, d_2=6n_2+2, d_3=6n_3+2, d_4=6n_4+2$ such that $n_1,n_2 \geq 1$ and $n_3, n_4= \displaystyle{\prod_i} p_i^2$ with $p_i$ a prime number.\\

We consider the rank 5 lattice,\\
$$h^2 \in M := <h^2,\alpha_1+(1,0,0),\alpha_2+(1,0,0),\alpha_3+(0,1,0), \alpha_4+(0,0,1)>$$

with the following Gram matrix:
$$\begin{pmatrix}
3&1&1 & 1&1\\
1&2n_1+1&0& 0&0 \\
1&0&2n_2+1&0&0\\
1&0&0&2n_3+1&\sqrt{n_3n_4}\\
1&0&0&\sqrt{n_3n_4}&2n_4+1\\
\end{pmatrix}$$
$M$ is a positive definite saturated sublattice of $L$.\\
For all $v \in M$,  $v=x_1h^2+x_2(\alpha_1+(1,0,0))+x_3(\alpha_2+(1,0,0))+x_4(\alpha_3+(0,1,0))+x_5(\alpha_4+(0,0,1)) \in M$, we have that
\begin{center}
    $<v,v>=2n_1x_2^2+2n_2x_3^2+n_3x_4^2+n_4x_5^2 + (\sqrt{n_3}x_4+\sqrt{n_4}x_5)^2 + (x_1+x_5)^2+ (x_1+x_4)^2+ (x_1+x_3)^2+(x_1+x_2)^2-x_1^2 \ne 2$
\end{center}
for all $n_1,n_2\geq 1$, $n_3, n_4= \displaystyle{\prod_i} p_i^2$ with $p_i$ a prime number. Then $\C_M$ is nonempty of codimension 4. We consider the rank 2 saturated sublattices \\
$$h^2 \in K_{d_i}:= <h^2, \alpha_i>, i:1,..,4.$$\\
Applying \cite[Lemma 2.4]{2019arXiv190501936Y}, we have that $\emptyset \ne \C_M \subset \C_{d_1} \cap \C_{d_2} \cap \C_{d_3} \cap \C_{d_4}$ with the required conditions on $d_i$ in this case.\\

\medskip
In the same way, one can continue to consider, for each $n$-intersection of Hassett divisors, certain positive definite saturated (n+1)-rank lattice $E$ and check that for a nonzero $v \in E$, $<v.v>$ is greater than 3 or equivalently  for all $v \in E$ $<v,v>$ is different of 2.\\

\medskip
Furthermore, in order to study the intersection of $20$ Hassett divisors $C_{d_k}$ such that all $d_k \equiv 2 [6]$, We consider, for example, the rank 21 lattice $M$, generated by  $h^2,\alpha_1+(0,1,0),\alpha_2+(1,0,0),\sqrt{n_3}a_1+(1,0,0),\sqrt{n_4}a_2+(0,1,0),\sqrt{n_5}t^1_1+(0,0,1),\sqrt{n_6}t^1_3+(0,0,1),\sqrt{n_7}t^1_6+(1,0,0),\sqrt{n_8}t^2_1+(0,1,0),\sqrt{n_9}t^2_3+(0,0,1),\sqrt{n_{10}}t^2_6+(0,0,1),\sqrt{n_{11}}t^1_2+(0,0,1),\sqrt{n_{12}}t^2_2+(0,0,1),\sqrt{n_{13}}t^1_4+(0,0,1),\sqrt{n_{14}}t^2_4+(1,0,0),\\ \sqrt{n_{15}}t^1_7+(0,1,0),\sqrt{n_{16}}t^2_7+(0,1,0),\sqrt{n_{17}}t^1_8+(0,1,0),\sqrt{n_{18}}t^2_8+(0,1,0),\sqrt{n_{19}}t^1_5+(0,1,0),\\ \sqrt{n_{20}}t^2_5+(0,1,0)$.\\

The Gram matrix of $M$ is:\\

$$\left(\begin{array}{ccc|ccc}
{}&A&{}&{}&B&{}\\
{}&{}&{}&{}&{}{}\\
\hline
{}&{}&{}&{}&{}{}\\
{}&B^T&{}&{}&C&{}
\end{array}\right)$$

with\\

\scalebox{0.8}
{$A=\begin{pmatrix}
3&1&1&1&1&1&1&1&1&1&1\\
1&2n_1+1&0& 0&1&0&0&0&1&0&0 \\
1&0&2n_2+1&1&0&0&0&1&0&0&0\\
1&0&1&2n_3+1&\sqrt{n_3n_4}&0&0&1&0&0&0\\
1&1&0&\sqrt{n_3n_4}&2n_4+1&0&0&0&1&0&0\\
1&0&0&0&0&2n_5+1&1&0&0&1&1\\
1&0&0&0&0&1&2n_6+1&0&0&1&1\\
1&0&1&1&0&0&0&2n_7+1&0&0&0\\
1&1&0&0&1&0&0&0&2n_8+1&0&0\\
1&0&0&0&0&1&1&0&0&2n_9+1&1\\
1&0&0&0&0&1&1&0&0&1&2n_{10}+1\\
\end{pmatrix}$}

\vspace{1cm}
\scalebox{0.8}
{$B=\begin{pmatrix} 

1&1&1&1&1&1&1&1&1&1\\
0&0&0&0&1&1&1&1&1&1\\
0&0&0&1&0&0&0&0&0&0\\
0&0&0&1&0&0&0&0&0&0\\
0&0&0&0&1&1&1&1&1&1\\
1-\sqrt{n_{5}n_{11}}&1&1&0&0&0&0&0&0&0\\
1-\sqrt{n_{6}n_{11}}&1&1-\sqrt{n_6n_{13}}&0&0&0&0&0&-\sqrt{n_6n_{19}}&0\\
0&0&0&1&-\sqrt{n_7n_{15}}&0&0&0&-\sqrt{n_7n_{19}}&0\\
0&-\sqrt{n_8n_{12}}&0&0&1&1&1&1&1&1\\
1&1-\sqrt{n_9n_{12}}&1&-\sqrt{n_9n_{14}}&0&0&0&0&0&-\sqrt{n_9n_{20}}\\
1&1&1&0&0&-\sqrt{n_{10}n_{16}}&0&0&0&-\sqrt{n_{10}n_{20}}\\
\end{pmatrix}$}\\

\vspace{1cm}
\scalebox{0.8}
{$C=\begin{pmatrix} 
2n_{11}+1&1&1&0&0&0&0&0&0&0\\
1&2n_{12}+1&1&0&0&0&0&0&0&0\\
1&1&2n_{13}+1&0&0&0&0&0&0&0\\
0&0&0&2n_{14}+1&0&0&0&0&0&0\\
0&0&0&0&2n_{15}+1&0&-\sqrt{n_{15}n_{17}}&0&0&0\\
0&0&0&0&0&2n_{16}+1&0&-\sqrt{n_{16}n_{18}}&0&0\\
0&0&0&0&-\sqrt{n_{15}n_{17}}&0&2n_{17}+1&0&0&0\\
0&0&0&0&0&-\sqrt{n_{16}n_{18}}&0&2n_{18}+1&0&0\\
0&0&0&0&0&0&0&0&2n_{19}+1&0\\
0&0&0&0&0&0&0&0&0&2n_{20}+1\\
\end{pmatrix}$}

\medskip
$M$ is a positive definite saturated sublattice of $L$. For all $v \in M, <v,v>
\ne 2$ $\forall n_1,n_2 \geq 1$, $n_k=\displaystyle{\prod_i}p_i^2$ with $p_i$ a prime number for $k=3,..,20$.\\ We consider rank 2 saturated sublattices $h^2 \in K_{d_i}$ of discriminant $d_i$ as before.\\ Therefore, by \cite[Lemma 2.4, Proposition 2.3]{2019arXiv190501936Y}, $\displaystyle{\bigcap_{k=1}^{20}} \C_{d_k} \ne \emptyset$ for all $d_k$ satisfying $(*)$ and $d_k$ satisfies $(**)$ for $k\geq 3$.
\end{proof}

One can choose all these divisors in a way such that their elements have associated K3 surfaces.
 
\begin{cor} \label{cor20}
Consider the rational cubic fourfolds $X$ inside the dimension 0 intersection $\C_{14} \cap \C_{38} \cap \C_{26} \cap \C_{98} \cap \C_{218} \cap \C_{294} \cap \C_{386} \cap \C_{602} \cap \C_{866} \cap \C_{1178} \cap \C_{1538} \cap \C_{1946} \cap \C_{2166} \cap \C_{2402} \cap \C_{2906} \cap \C_{3458} \cap \C_{4058} \cap \C_{4706} \cap \C_{6146} \cap \C_{6938}$ .

Now, cubic fourfolds in each of these divisors have associated polarized K3 surfaces. Hence, the K3 surface $S$ associated to each cubic in the intersection has 20 different polarizations, that is $rk(NS(S))=20$.
\end{cor}

\begin{rem}
Equivalently, one can say that every $X$ mentioned in Corollary \ref{cor20} has an associated singular K3 surface.
\end{rem}
Additionally, one can conjecture the following:\\

\medskip
\textbf{Conjecture:} A divisor with discriminant $d=6.2^{2k} \displaystyle{\prod_i} p_i^2+2$ for $k \in \mathbb{N}^*$, $p_i$ a prime number, has always an associated K3 surface.

\begin{rem}
To prove the conjecture, one can see easily that, for $d=6.2^{2k} \displaystyle{\prod_i} p_i^2+2$, $4 \not | d$ and $9 \not | d$. However, it seems to be harder to see that for any odd prime number $p \equiv 2[3], p \not |d$. \end{rem}

\bibliography{bib_tocho}
\bibliographystyle{abbrv}

\end{document}